\documentclass{article}
\usepackage{amsmath}
\usepackage{amsthm}
\usepackage{amsfonts}
\usepackage{amssymb}
\usepackage{graphicx}
\usepackage{color}

\newtheorem{definition}{Definition}
\newtheorem{thm}{Theorem}

\newtheorem{remark}{Remark}
\newtheorem{corollary}{Corollary}

\allowdisplaybreaks

\numberwithin{equation}{section}

\title{FRACTIONAL SPHERICAL RANDOM FIELDS}

\author{
Mirko D'Ovidio\footnote{Department of Basic and Applied Sciences for Engineering, Sapienza
University of Rome, Via A. Scarpa, 16, 00161, Roma, Italy}, 
Nikolai Leonenko\footnote{School of Mathematics, Cardiff University, Senghennydd Road, Cardiff CF24 4YH, UK},
Enzo Orsingher\footnote{Department of Statistical Sciences, Sapienza University of Rome, P.le Aldo Moro, 5, 00185, Rome, Italy}
}

\begin{document}

\maketitle

\begin{abstract}
 In this paper we study the solutions of different forms of
fractional equations on the unit sphere $\mathbb{S}_{1}^{2}$ $\subset 
\mathbb{R}^{3}$ possessing the structure of time-dependent random fields. We
study the correlation functions of the random fields emerging in the
analysis of the solutions of the fractional equations and examine their
long-range behaviour.
\end{abstract}



\section{Introduction}

In this paper we deal with various forms of random fields on the unit sphere 
$\mathbf{S}_{1}^{2}$ indexed by the spherical Brownian motion. We restrict
ourselves to isotropic random fields for which the expansion in terms of
spherical harmonics holds (see \cite{DomPec-book} and the references
therein). The explicit law of the Brownian motion on $\mathbf{S}_{1}^{2}$
was first obtained in \cite{Yios1949}. For Brownian motion on $\mathbf{S}%
_{1}^{d}$, see \cite[pag. 338]{KarlinTaylor}. Time-dependent random fields
on the line or on arbitrary Euclidean spaces have been studied by several
authors (see, for example, \cite{KLRM,Angulo08,LRMT} and the
references therein). We here study time-dependent random fields on the
sphere $\mathbf{S}_{1}^{2}$, governed by different stochastic differential
equations.

We first study random fields emerging from the Cauchy problem 
\begin{equation}
\left\{ 
\begin{array}{ll}
\displaystyle\left( \gamma -\mathbb{D}_{M}+\frac{\partial ^{\beta }}{%
\partial t^{\beta }}\right) X_{t}(x)=0, & x\in \mathbf{S}_{1}^{2},\;t>0,\;0<%
\beta <1,\;\gamma >0 \\ 
\displaystyle X_{0}(x)=T(x), & 
\end{array}%
\right.   \label{eq1-intro}
\end{equation}%
where $\mathbb{D}_{M}$ is a suitable differential operator defined below, $%
\frac{\partial ^{\beta }}{\partial t^{\beta }}$ is the Dzerbayshan-Caputo
fractional derivative. By $T(x)$, $x\in \mathbf{S}_{1}^{2}$ we denote an
isotropic Gaussian field on the unit sphere. We are able to obtain the
solution $X_{t}(x)$ of \eqref{eq1-intro} and to show that its covariance
function displays a long-memory behaviour.

We then consider the non-homogeneous fractional equation 
\begin{equation}
\left( \gamma -\mathbb{D}_{M}\right) ^{\beta }X(x)=T(x),\quad x\in \mathbf{S}%
_{1}^{2},\;0<\beta <1
\end{equation}%
of which 
\begin{equation}
\left( \gamma -\mathbb{D}_{M}-\varphi \frac{\partial }{\partial t}\right)
^{\beta }X_{t}(x)=T_{t}(x),\quad x\in \mathbf{S}_{1}^{2},\;t>0,\;0<\beta
<1,\;\gamma >0,\;\varphi \geq 0 \label{eq2-intro}
\end{equation}%
is the time-dependent extension. We obtain a solution to \eqref{eq2-intro}
which is a random field on the sphere with covariance function with a
short-range dependence.

The couple $(B_{s}(t),T(x+B_{s}(t)))$ describes a random motion on the
unit-radius sphere with dynamics governed by fractional stochastic equations %
\eqref{eq1-intro} and \eqref{eq2-intro}.

Random fields similar to those examined here are considered in the analysis
of the cosmic microwave background radiation (CMB radiation). In this case,
the correlation structure turns out to be very important as well as the
angular power spectrum. The angular power spectrum plays a key role in the
study of the corresponding random field. In particular, the high-frequency
behaviour of the angular power spectrum is related to some anisotropies of
the CMB radiation (see for example \cite{DovCoordinates, DomPec-book}). Such
relations have been also investigated in \cite{DovNane} where a coordinates
change driven by a fractional equation has been considered.

Diffusions on the sphere arise in several contexts. At the
cellular level, diffusion is an important mode of transport of substances.
The cell wall is a lipid membrane and biological substances like lipids and
proteins diffuse on it. In general biological membranes are curved surfaces.
Spherical diffusions also crop up in the swimming of bacteria, surface
smoothening in computer graphics [\cite{BU}] and global migration patterns
of marine mammals \cite{BR}. 

\section{Preliminaries}

\subsection{Isotropic random fields on the unit-radius sphere}

We consider the square integrable $2$-weakly isotropic Gaussian random field 
\begin{equation}
\{T(x);\,x\in \mathbf{S}_{1}^{2}\}
\end{equation}%
on the sphere $\mathbf{S}_{1}^{2}=\{x\in \mathbf{R}^{3}\,:\,|x|=1\}$ for
which 
\begin{eqnarray*}
\mathbb{E}T(gx) &=&0, \\
\mathbb{E}T^{2}(gx) &=&\mathbb{E}T^{2}(x) \\
\mathbb{E}[T(gx_{1})\,T(gx_{2})] &=&\mathbb{E}[T(x_{1})\,T(x_{2})].
\end{eqnarray*}%
for all $g\in SO(3)$ where $SO(3)$ is the special group of rotations in $%
\mathbf{R}^{3}$. We will consider the spectral representation 
\begin{equation}
T(x)=\sum_{l=0}^{\infty }\sum_{m=-l}^{+l}a_{l,m}\mathcal{Y}%
_{l,m}(x)=\sum_{l=0}^{\infty }T_{l}(x)  \label{Trep}
\end{equation}%
where 
\begin{equation}
a_{l,m}=\int_{\mathbf{S}^{2}}T(x)\mathcal{Y}_{l,m}^{\ast }(x)\lambda
(dx),\quad -l\leq m\leq +l,\;l\geq 0  \label{alm-coeff-intro}
\end{equation}%
are the Fourier random coefficients of $T$. The convergence in \eqref{Trep}
must be meant in the sense that 
\begin{equation}
\lim_{L\rightarrow \infty }\mathbb{E}\left[ \int_{\mathbf{S}_{1}^{2}}\left(
T(x)-\sum_{l=0}^{L}\sum_{m=-l}^{+l}a_{l,m}\,\mathcal{Y}_{l,m}(x)\right)
^{2}\lambda (dx)\right] =0  \label{covergence-series-one}
\end{equation}%
where $\lambda (dx)$ is the Lebesgue measure on the sphere $\mathbf{S}%
_{1}^{2}$, $\{\mathcal{Y}_{l,m}(x):\,l\geq 0,\;m=-l,\ldots ,+l,\;x\in 
\mathbf{S}_{1}^{2}\}$ is the set of spherical harmonics representing an
orthonormal basis for the space $L^{2}(\mathbf{S}_{1}^{2},\lambda (dx))$. By 
$\mathcal{Y}_{l,m}^{\ast }(x)$ we denote the conjugate of $\mathcal{Y}%
_{l,m}(x)$. For the sake of clarity we observe that for all $x\in \mathbf{S}%
_{1}^{2}$ and $0\leq \vartheta \leq \pi ,\;0\leq \varphi \leq 2\pi $: 
\begin{equation*}
\lambda (dx)=\lambda (d\vartheta ,d\varphi )=d\varphi \,d\vartheta \,\sin
\vartheta 
\end{equation*}%
and 
\begin{equation*}
x=(\sin \vartheta \cos \varphi ,\sin \vartheta \sin \varphi ,\cos \vartheta
).
\end{equation*}%
We shall write $f(x)$ instead of $f(\vartheta ,\varphi )$ when no confusion
arises.

The random coefficients \eqref{alm-coeff-intro} are zero-mean Gaussian
complex random variables such that (\cite{Baldi-Marinucci-2007}) 
\begin{equation}
\mathbb{E}[a_{l,m}\;a^*_{l^\prime, m^\prime}] =
\delta_l^{l^\prime}\delta_m^{m^\prime} \mathbb{E} | a_{l,m} |^2
\label{angular-power-C}
\end{equation}
where 
\begin{equation}
\mathbb{E}| a_{l,m} |^2 = C_l, \quad l \geq 0
\end{equation}
is the angular power spectrum of the random field $T$ which under the
assumption of Gaussianity fully characterizes the dependence structure of $T$%
. Clearly, $\delta_a^b$ is the Kronecker symbol.

For a fixed integer $l$ we define $\mu _{l}=l(l+1)$. The spherical harmonics 
$Y_{l,m}(\vartheta ,\varphi )$ are defined as 
\begin{equation*}
\mathcal{Y}_{l,m}(\vartheta ,\varphi )=\sqrt{\frac{2l+1}{4\pi }\frac{(l-m)!}{%
(l+m)!}}Q_{l,m}(\cos \vartheta )e^{im\varphi },\quad 0\leq \vartheta \leq
\pi ,\;0\leq \varphi \leq 2\pi 
\end{equation*}%
where 
\begin{equation*}
Q_{l,m}(z)=(-1)^{m}(1-z^{2})^{\frac{m}{2}}\frac{d^{m}}{dz^{m}}Q_{l}(z),\quad
|z|<1
\end{equation*}%
are the associated Legendre functions and $Q_{l}$ are the the Legendre
polynomials with Rodrigues representation 
\begin{equation*}
Q_{l}(z)=\frac{1}{2^{l}l!}\frac{d^{l}}{dz^{l}}(z^{2}-1)^{l},\quad |z|<1.
\end{equation*}%
We remind that the spherical harmonics solve 
\begin{equation}
\Delta _{\mathbf{S}_{1}^{2}}\mathcal{Y}_{l,m}=-\mu _{l}\,\mathcal{Y}%
_{l,m},\quad l\geq 0,\;|m|\leq l  \label{eigenY}
\end{equation}%
where 
\begin{equation*}
\Delta _{\mathbf{S}_{1}^{2}}=\frac{1}{\sin ^{2}\vartheta }\frac{\partial ^{2}%
}{\partial \varphi ^{2}}+\frac{1}{\sin \vartheta }\frac{\partial }{\partial
\vartheta }\left( \sin \vartheta \frac{\partial }{\partial \vartheta }%
\right) 
\end{equation*}%
is the spherical Laplace operator or Laplace-Beltrami operator.

In view of \eqref{angular-power-C}, the covariance function of $T(x)$ writes 
\begin{equation}
\mathbb{E}[T(x)T(y)] = \sum_{lm} C_l \mathcal{Y}_{l,m}(x)\mathcal{Y}%
_{l,m}^*(y) = \sum_{l} C_l \frac{2l+1}{4\pi}Q_l(\langle x, y \rangle)
\end{equation}
where in the last step we used the addition formula for spherical harmonics 
\begin{equation}
\sum_{m=-l}^{+l} \mathcal{Y}_{l,m}(y)\mathcal{Y}^*_{l,m}(x) = \frac{2l+1}{%
4\pi} Q_l(\langle x, y \rangle).  \label{addition-f}
\end{equation}
and the inner product 
\begin{equation*}
\langle x,y \rangle = \cos d(x,y)= \cos \vartheta_x \cos
\vartheta_y+\sin\vartheta_x \sin\vartheta_y\cos (\varphi_x-\varphi_y)
\end{equation*}
where $d(\cdot, \cdot)$ is the spherical distance between the points $x,y$.

For the details on this material we refer to the book by Marinucci and
Peccati \cite{DomPec-book}.

\subsection{Subordinators and fractional operators}

Let  $F(t)$, $t\geq 0$ be a L\'{e}vy subordinator with characteristic
function 
\begin{equation}
\mathbb{E}e^{i\xi F(t)}=e^{-t\Phi (\xi )}=e^{-t\left( ib\xi
+\int_{0}^{\infty }\left( e^{i\xi y}-1\right) M(dy)\right) },
\end{equation}%
where $b\geq 0$ is the drift and $M(\cdot )$ is the L\'{e}vy measure $M$ on $%
\mathbf{R}_{+}\setminus \{0\}$ satisfying the condition:%
\begin{equation*}
\int_{0}^{\infty }(y\wedge 1)M(dy)<\infty .
\end{equation*}
$\int_{0}^{\infty }(y\wedge 1)M(dy)<\infty $ and $M(-\infty ,0)=0$. The
Laplace transform of the law of a subordinator $F(t)$, $t>0$ defined above,
can be written as 
\begin{equation}
\mathbb{E}e^{-\xi F(t)}=e^{-t\Psi (\xi )}=e^{t\Phi (i\xi )}=e^{-t\left( b\xi
+\int_{0}^{\infty }\left( 1-e^{-\xi y}\right) M(dy)\right) },
\label{lap-exp-sub}
\end{equation}%
where $\Psi (\xi )$ is known as Laplace exponent. If $F(t)$, $t\geq 0$, is
the $\beta$-stable subordinator, then $\Psi (\xi )=\xi ^{\beta }$, $\beta
\in (0,1)$. Hereafter, we assume $b=0$.

We write the transition density of a Brownian motion on the unit sphere (see 
\cite{Yios1949}) as follows 
\begin{eqnarray}
Pr\{x+B_{t}\in dy\}/dy &=&Pr\{B_{t}\in dy\,|\,B_{0}=x\}/dy  \notag \\
&=&\sum_{l=0}^{\infty }\sum_{m=-l}^{+l}e^{-t\mu _{l}}\mathcal{Y}_{l,m}(y)%
\mathcal{Y}_{l,m}^{\ast }(x)  \notag \\
&=&\sum_{l}e^{-t\mu _{l}}\frac{2l+1}{4\pi }Q_{l}(\langle x,y\rangle )
\label{PdiB}
\end{eqnarray}%
where we used the addition formula for spherical harmonics \eqref{addition-f}%
. Furthermore, we shall write 
\begin{equation}
P_{t}f(x)=\mathbb{E}f(x+B_{t})=\int_{\mathbf{S}_{1}^{2}}f(y)Pr\{x+B_{t}\in
dy\}  \label{semigB}
\end{equation}%
where $P_{t}f(x)$ is the solution to the initial-value problem 
\begin{equation}
\left\{ 
\begin{array}{ll}
\displaystyle\frac{\partial u}{\partial t}=\Delta _{\mathbf{S}_{1}^{2}}u, & 
x\in \mathbf{S}_{1}^{2},\,t>0 \\ 
\displaystyle u(x,0)=f(x) & 
\end{array}%
\right. 
\end{equation}%
for a measurable function $f(x),x\in \mathbf{S}_{1}^{2}.$

Let $f$ be a square integrable function on the unit sphere, that is $f \in
L^2(\mathbf{S}^2_1)$. We define the following operator 
\begin{equation}
\mathbb{D}_{M} f(x) = \int_0^\infty \left( P_s\, f(x) - f(x) \right) M(ds)
\label{frac-oper-sphere}
\end{equation}
where, from \eqref{PdiB} and \eqref{semigB}, we have that 
\begin{equation}
P_s f(x) = \sum_{l=0}^\infty \sum_{m=-l}^{+l} e^{-s\mu_l} \mathcal{Y}%
_{l,m}(x) f_{l,m}  \label{explicit-semigB}
\end{equation}
and $f_{l,m}$ are the Fourier coefficients of $f$. The operator %
\eqref{frac-oper-sphere} can be rewritten as 
\begin{equation}
\mathbb{D}_M\, f(x) = \int_{\mathbf{S}^2_1}\left( f(y) - f(x) \right) 
\widehat{J}(x,y)\lambda(dy)  \label{D-op-rewritten}
\end{equation}
where $\lambda$ is the Lebesgue measure on $\mathbf{S}^2_1$ and 
\begin{equation*}
\widehat{J}(x,y) = \sum_{l=0}^\infty \frac{2l+1}{4\pi} Q_l(\langle y, x
\rangle) \widehat{\Psi}(\mu_l)
\end{equation*}
with $\widehat{\Psi}(\mu) = \int_0^\infty e^{-s \mu} M(ds)$ when the
integral exists. Indeed we can write 
\begin{align*}
\mathbb{D}_M\, f(x) = & \int_0^\infty \left( P_sf(x) - f(x) \right) M(ds) \\
= & \int_0^\infty \mathbb{E}\left[\left( f(x+B_s) - f(x) \right) \right]
M(ds) \\
= & \int_0^\infty \int_{\mathbf{S}^2_1}\left( f(y) - f(x) \right) Pr\{ x+
B_s \in dy \} M(ds) \\
= & \int_{\mathbf{S}^2_1}\left( f(y) - f(x) \right) \widehat{J}%
(x,y)\lambda(dy)
\end{align*}
where 
\begin{align*}
\widehat{J}(x,y) \lambda(dy)= & \int_0^\infty Pr\{ x+B_s \in dy \} M(ds) \\
= & \lambda(dy) \sum_{l} \frac{2l+1}{4\pi} Q_l(\langle y, x \rangle)
\int_0^\infty e^{-s \mu_l} M(ds) \\
= & \lambda(dy) \sum_{l} \frac{2l+1}{4\pi} Q_l(\langle y, x \rangle) 
\widehat{\Psi}(\mu_l).
\end{align*}
Furthermore, from \eqref{explicit-semigB}, the operator %
\eqref{frac-oper-sphere} can be written as follows 
\begin{align*}
\mathbb{D}_M f(x) = & \int_0^\infty \left( P_s f(x) - P_0f(x) \right) M(ds)
\\
= & \sum_{lm} f_{l,m} \mathcal{Y}_{l,m}(x) \int_0^\infty \left( e^{-s \mu_l}
- 1 \right)M(ds) \\
= & \left[ by \, (\ref{lap-exp-sub}) \right] = -\sum_{lm} f_{l,m} \mathcal{Y}%
_{l,m}(x) \Psi(\mu_l) \\
= & -\sum_{lm} \left( \int_{\mathbf{S}^2_1} f(y) Y^*_{l,m}(y) \lambda(dy)
\right) \mathcal{Y}_{l,m}(x) \Psi(\mu_l) \\
= & -\int_{\mathbf{S}^2_1} f(y) \left( \sum_{lm} \Psi(\mu_l) \mathcal{Y}%
_{l,m}(x) \mathcal{Y}_{l,m}^*(y) \right) \lambda(dy) \\
= & -\int_{\mathbf{S}^2_1} f(y) J(x,y) \lambda(dy)
\end{align*}
where 
\begin{equation}
J(x,y) = \sum_{l=0}^\infty \sum_{m=-l}^{+l} \Psi(\mu_l) \mathcal{Y}_{l,m}(x) 
\mathcal{Y}_{l,m}^*(y) = \sum_{l=0}^\infty \Psi(\mu_l) \frac{2l+1}{4\pi}
Q_l(\langle x, y \rangle)
\end{equation}
exists (in the last step we have applied the addition formula %
\eqref{addition-f}).

We introduce the Sobolev space 
\begin{equation}
H^{s}(\mathbf{S}^2_1) = \left\lbrace f \in L^2(\mathbf{S}^2_1):\,
\sum_{l=0}^{\infty} (2l+1)^{2s} f_l < \infty \right\rbrace  \label{SobolevS}
\end{equation}
where 
\begin{equation*}
f_l= \sum_{|m| \leq l} \big| f_{l,m} \big|^2 = \sum_{|m| \leq l} \Bigg| %
\int_{\mathbf{S}^2_1} f(x) \mathcal{Y}^*_{l,m}(x)\lambda(dx) \Bigg|^2, \quad
l=0,1,2, \ldots .
\end{equation*}

\begin{definition}
Let $\Psi$ be the symbol of a subordinator. Let $f \in H^s(\mathbf{S}^2_1)$
and $s>4$. Then, 
\begin{equation}
\mathbb{D}_M f(x) = -\sum_{l=0}^\infty \sum_{m=-l}^{+l} f_{l,m} \mathcal{Y}%
_{l,m}(x) \Psi(\mu_l)  \label{def-der-psi}
\end{equation}
where 
\begin{equation*}
f_{l,m} = \int_{\mathbb{S}^2_1} f(x) \mathcal{Y}^*_{l,m}(x) \lambda(dx)
\end{equation*}
are the Fourier coefficients of the initial condition.
\end{definition}

The series \eqref{def-der-psi} converges absolutely and uniformly. Indeed, $%
f_l < l^{-2s}$ with $s> 4$ (being $f \in H^s(\mathbf{S}^2_1)$), $\|
Y_{l,m}\|_{\infty} < l^{1/2}$ (see \cite{Quantum}) and $\Psi(\mu_l) \leq l^2$
and thus, by considering that 
\begin{equation*}
\sum_{m} |f_{l,m}| \leq \left( \sum_{m} |f_{l,m}|^2 \right)^\frac{1}{2}
\left( (2l+1) \right)^\frac{1}{2} = \sqrt{(2l+1) \, f_l} \leq l^{-s + \frac{1%
}{2}}
\end{equation*}
we get the claim.

\begin{definition}
$\mathbb{P}_t = \exp(t \mathbb{D}_M)$ is the semigroup associated with %
\eqref{D-op-rewritten} with symbol $\widehat{\mathbb{P}_t} = \exp(-t \Psi)$
where $-\Psi$ is the Fourier multiplier of $\mathbb{D}_M$.
\end{definition}

\section{Some fractional equations on the sphere}

We recall the Dzerbayshan-Caputo fractional derivative 
\begin{equation}
\frac{\partial ^{\beta }u}{\partial t^{\beta }}(x,t)=\frac{1}{\Gamma
(1-\beta )}\int_{0}^{t}\frac{\partial u(x,s)}{\partial s}\frac{ds}{%
(t-s)^{\beta }}
\end{equation}%
for $0<\beta <1$, $x\in \mathbb{R}$, $t>0$, see, e.g., \cite{MS}, p. 38.

The inverse $\mathfrak{L}_{t}^{\beta }$ of a $\beta -$stable subordinator $%
\mathfrak{H}_{t}^{\beta }$ can be defined by the following relationship 
\begin{equation*}
Pr\{\mathfrak{L}_{t}^{\beta }<x\}=Pr\{\mathfrak{H}_{x}^{\beta }>t\}
\end{equation*}%
for $x,t>0$, see, e.g., \cite{MS}, p. 101.

The Mittag-Leffler function is defined as 
\begin{equation}
E_{\beta }(x)=\sum_{k=0}^{\infty }\frac{x^{k}}{\Gamma (\beta k+1)},\quad
x\in \mathbb{R},\;\beta >0,
\end{equation}%
see, e.g.,  \cite{MS}, p. 35.

We assume also that the random field $T$ introduced in \eqref{Trep} is
Gaussian and its Fourier coefficients $a_{l,m}$ are independent complex
zero-mean Gaussian r.v.'s. We shall use the following notation 
\begin{equation*}
\sum_{l=0}^{\infty }\sum_{m=-l}^{+l}=\sum_{lm}
\end{equation*}%
when no confusion arises.

We pass now to the first theorem. Denote by $F^\Psi(\mathfrak{L}^\beta_t)$
the subordinator with symbol $\Psi$ time-changed by the inverse of a stable
subordinator of order $\beta \in (0,1)$.

\begin{thm}
Let us consider $\gamma \geq 0$ and $\beta \in (0,1)$. The solution to the
fractional equation 
\begin{equation}
\left( \gamma -\mathbb{D}_{M}+\frac{\partial ^{\beta }}{\partial t^{\beta }}%
\right) \,X_{t}(x)=0,\quad x\in \mathbf{S}_{1}^{2},\;t\geq 0
\label{rand-pde}
\end{equation}%
with initial condition $X_{0}(x)=T(x)$ is a time-dependent random field on
the sphere $\mathbf{S}_{1}^{2}$ written as 
\begin{equation}
X_{t}(x)=\sum_{l=0}^{\infty }\sum_{m=-l}^{+l}a_{l,m}E_{\beta }\left(
-t^{\beta }(\gamma +\Psi (\mu _{l}))\right) \mathcal{Y}_{l,m}(x)
\label{X-sol-gammaPDE}
\end{equation}%
where 
\begin{equation}
a_{l,m}=\int_{\mathbf{S}_{1}^{2}}X_{0}(x)\mathcal{Y}_{l,m}^{\ast }(x)\lambda
(dx).
\end{equation}%
Furthermore, the following representation holds 
\begin{equation}
X_{t}(x)=\mathbb{E}\left[ T(x+B(\gamma \mathfrak{L}_{t}^{\beta }+F^{\Psi }(%
\mathfrak{L}_{t}^{\beta })))\big|\mathfrak{F}_{T}\right] 
\label{X-rep-F-algebra}
\end{equation}%
where $\mathfrak{F}_{T}$ is the $\sigma $-field generated by $X_{0}=T$. %
\label{thm:one}
\end{thm}

\begin{proof}
First we notice that
\begin{equation}
- \frac{\partial}{\partial t}\, \mathbb{E} e^{-\xi (\gamma t + F_t)} \Big|%
_{t=0} = \xi \gamma + \Psi(\xi)
\end{equation}
which coincides with \eqref{lap-exp-sub} for $\gamma=b$. Indeed, we are
dealing with the symbol $\Psi$ of the subordinator $F_t$ without drift. Furthermore, it is well-known that the Mittag-Leffler function $E_{\beta}$ is an eigenfunction of the Dzerbayshan-Caputo fractional derivative, that is
\begin{equation}
\frac{\partial^\beta}{\partial t^\beta} E_{\beta}(-t^\beta \mu) = -\mu E_{\beta}(-t^\beta \mu) \label{eigen-E}.
\end{equation}

We assume that \eqref{X-sol-gammaPDE} holds true. From the fact that
\begin{align*}
\mathbb{D}_M \, \mathcal{Y}_{l,m}(x) = \int_0^\infty \left( P_s \mathcal{Y}_{l,m}(x) -
\mathcal{Y}_{l,m}(x) \right) M(ds)
\end{align*}
where $\mathcal{Y}_{l,m}(x) = (-1)^m \mathcal{Y}^*_{l,-m}(x)$ and
\begin{equation}
P_s \mathcal{Y}_{l,m}(x) = e^{-s \mu_{l}} \mathcal{Y}_{l,m}(x) \label{PofB}
\end{equation}
we obtain that
\begin{align*}
\mathbb{D}_M \, \mathcal{Y}_{l,m}(x) = & \int_0^\infty \left( e^{-s \mu_l}
\mathcal{Y}_{l,m}(x) - \mathcal{Y}_{l,m}(x) \right) M(ds) \\
= & \int_0^\infty \left( e^{-s \mu_l} - 1 \right) M(ds) \, \mathcal{Y}_{l,m}(x) \\
= & - \Psi(\mu_l )\, \mathcal{Y}_{l,m}(x).
\end{align*}
Formula \eqref{PofB} can be obtained by considering that
\begin{align*}
P_s \mathcal{Y}_{l,m}(x) = & \mathbb{E} \mathcal{Y}_{l,m}(x + B_s)   \\
= & \sum_{l^\prime m^\prime} e^{-s \mu_{l^\prime}} \mathcal{Y}^*_{l^\prime, m^\prime}(x) \int_{\mathbf{S}^2_1} \mathcal{Y}_{l,m}(y)\mathcal{Y}_{l^\prime, m^\prime}\lambda(dy) \\
= & \sum_{l^\prime m^\prime} e^{-s \mu_{l^\prime}} \mathcal{Y}^*_{l^\prime, m^\prime}(x) (-1)^{m^\prime} \int_{\mathbf{S}^2_1} \mathcal{Y}_{l,m}(y)\mathcal{Y}^*_{l^\prime, -m^\prime}\lambda(dy)   \\
= & \sum_{l^\prime m^\prime} e^{-s \mu_{l^\prime}} \mathcal{Y}^*_{l^\prime, m^\prime}(x) (-1)^{m^\prime} \delta_l^{l^\prime}\delta_m^{-m^\prime}= e^{-s \mu_{l}} \mathcal{Y}_{l,m}(x) . 
\end{align*}
Thus, we get that
\begin{align*}
\left( \gamma - \mathbb{D}_M \right) X_t(x) = & \sum_{l=0}^\infty
\sum_{m=-l}^{+l} a_{l,m} \left( \gamma + \Psi(\mu_l) \right) E_{\beta,1} \left( -t^\beta \gamma -
t^\beta \Psi(\mu_l) \right) \mathcal{Y}_{l,m}(x)
\end{align*}
and, from \eqref{eigen-E}, we arrive at
\begin{align*}
& \left( \frac{\partial^\beta}{\partial t^\beta} + \gamma - \mathbb{D}_M
\right)\, X_t(x) = \\
= & \sum_{l=0}^\infty \sum_{m=-l}^{+l} a_{l,m} \left( \frac{\partial^\beta}{%
\partial t^\beta} + \gamma +\Psi(\mu_l) \right) E_{\beta,1} \left(-t^\beta \gamma - t^\beta \Psi(\mu_l) \right)
\mathcal{Y}_{l,m}(x) = 0
\end{align*}
term by term and therefore equation \eqref{rand-pde} is satisfied. This concludes the proof.
\end{proof}

\begin{remark}
Since 
\begin{equation}
T(x) = \sum_{lm} a_{l,m} \mathcal{Y}_{l,m}(x).
\end{equation}
we have that 
\begin{equation}
P_tT(x) = \mathbb{E}[T(x+B_t)| \mathfrak{F}_T] = \sum_{lm} e^{- t \mu_l }
a_{l,m} \mathcal{Y}_{l,m}(x) = T_t(x).  \label{etimedepinn}
\end{equation}
This represents the solution to \eqref{rand-pde} with $\beta=1$, $\gamma=0$
and $\mathbb{D}_M = \Delta_{\mathbb{S}^2_1}$. From \eqref{X-rep-F-algebra},
for $\beta=1$, $\Psi(\xi)=\xi$, that is for the elementary subordinator $%
F(t)=t$ (and $\mathfrak{L}^1_t=t$) we have that 
\begin{align*}
X_t(x) = & \mathbb{E}\left[ T(x+B(\gamma t + t)) \big| \mathfrak{F}_T \right]
= \sum_{lm} a_{l,m} e^{-t(\gamma+ 1) \mu_l} \mathcal{Y}_{l,m}(x).
\end{align*}
\end{remark}

\begin{remark}
The series \eqref{X-sol-gammaPDE} converges in $L^{2}(dP\times d\lambda )$
sense for all $t\geq 0$, that is 
\begin{equation}
\lim_{L\rightarrow \infty }\mathbb{E}\left[ \int_{\mathbf{S}_{1}^{2}}\left(
X_{t}(x)-\sum_{l=0}^{L}\sum_{m=-l}^{+l}a_{l,m}\,\mathcal{T}_{l}(t)\,\mathcal{%
Y}_{l,m}(x)\right) ^{2}\lambda (dx)\right] =0,\quad \forall \,t.
\label{covergence-series}
\end{equation}%
where, in formula \eqref{X-sol-gammaPDE}, the time-dependent random
coefficients are 
\begin{equation*}
\mathcal{T}_{l}(t)=E_{\beta }\left( -t^{\beta }(\gamma +\Psi (\mu
_{l}))\right) ,\quad l\geq 0.
\end{equation*}%
Throughout the paper the convergence \eqref{covergence-series} in mean
square sense on $\mathbb{S}_{1}^{2}$ is considered. \label%
{remark-form-convergence}
\end{remark}


\begin{thm}
Let us consider $\gamma, \varphi \geq 0$ and $\beta \in (0,1]$. A solution
to the fractional equation 
\begin{equation}
\left( \gamma - \mathbb{D}_M - \varphi \frac{\partial}{\partial t}
\right)^{\beta }\, X_t(x) = T_t(x), \quad x \in \mathbf{S}^2_1, \; t\geq 0
\label{rand-pde-varphi}
\end{equation}
where $T_t(x)$ is given in \eqref{etimedepinn}, is a time-dependent random
field on the sphere $\mathbf{S}^2_1$ written as 
\begin{equation}
X_t(x) = \sum_{lm} a_{l,m} e^{-t \mu_l} \left( \gamma + \varphi \mu_l +
\Psi(\mu_l) \right)^{-\beta} \mathcal{Y}_{l,m}(x),
\end{equation}
where $e^{-t \mu_l}a_{l,m}$ are the Fourier coefficients involved in the
representation (\ref{etimedepinn}) of the innovation process $T_t(x)$ in (%
\ref{etimedepinn}) in terms of spherical harmonics. \label{thm:two}
\end{thm}

\begin{proof}
\label{thm-pro}
We have that
\begin{align*}
X_t(x) = & \left( \gamma - \mathbb{D}_M - \varphi \frac{\partial}{%
\partial t} \right)^{-\beta}\, T_t(x) \\
= & \int_0^\infty ds\, \frac{s^{\beta-1}}{\Gamma(\beta)} e^{s \varphi \frac{%
\partial}{\partial t} - s\gamma + s \mathbb{D}_M } T_t(x) \\
= & \int_0^\infty ds\, \frac{s^{\beta-1}}{\Gamma(\beta)} e^{s \varphi \frac{%
\partial}{\partial t} - s\gamma} \mathbb{P}_s T_t(x) \\
= & \int_0^\infty ds\, \frac{s^{\beta-1}}{\Gamma(\beta)} e^{- s\gamma}
\mathbb{P}_s T_{t+\varphi s}(x)
\end{align*}
where we used the translation rule
\begin{equation*}
e^{a\frac{\partial}{\partial z}}f(z)=f(z+a), \quad a \in \mathbb{R}
\end{equation*}
\noindent which holds for bounded continuous functions $f$ on $(0,
+\infty)$ (see, for example, formula (3.9) in \cite{DovWright} and the
references therein for details). From the fact that
\begin{equation}
\mathbb{P}_s \mathcal{Y}_{l,m}(x) = e^{-s\Psi(\mu_l)} \mathcal{Y}_{l,m}(x)
\end{equation}
\noindent where $\mathbb{P}_s=\exp(s\mathbb{D}_M)$ we get
that
\begin{align*}
X_t(x) = & \sum_{lm} a_{l,m} \left( \int_0^\infty ds\, \frac{s^{\beta-1}}{%
\Gamma(\beta)} e^{- s\gamma} e^{-(t+\varphi s) \mu_l} \mathbb{P}_s \mathcal{Y}_{l,m}(x)
\right) \\
= & \sum_{lm} a_{l,m} \left( \int_0^\infty ds\, \frac{s^{\beta-1}}{%
\Gamma(\beta)} e^{- s\gamma} e^{-(t+\varphi s)\mu_l} e^{-s\Psi(\mu_l)}
\right) \mathcal{Y}_{l,m}(x) \\
= & \sum_{lm} a_{l,m} e^{-t \mu_l}\left( \int_0^\infty ds\, \frac{s^{\beta-1}%
}{\Gamma(\beta)} e^{- s\gamma - s \varphi \mu_l -s\Psi(\mu_l)} \right)
\mathcal{Y}_{l,m}(x) \\
= & \sum_{lm} a_{l,m} e^{-t \mu_l} \left( \gamma + \varphi \mu_l +
\Psi(\mu_l) \right)^{-\beta} \mathcal{Y}_{l,m}(x)
\end{align*}
and this concludes the proof.
\end{proof}


We now examine the special case $\varphi=0$.

\begin{corollary}
Let $\alpha \in (0,1)$, $\beta \in (0,1]$. The solution to 
\begin{equation}
\left( \gamma - \mathbb{D}_M \right)^{\beta }X(x) = T(x)  \label{gamma-D-X}
\end{equation}
is written as 
\begin{equation}
X(x) = \sum_{lm} a_{l,m} \left(\gamma + \Psi(\mu_l)\right)^{-\beta} \mathcal{%
Y}_{l,m}(x).  \label{confirm-sol}
\end{equation}
\label{coro:one}
\end{corollary}

\begin{proof}
For $\beta \in (0,1)$ we consider the following relation concerning the
fractional power of operators (Bessel potential). For $f \in L^2(\mathbf{S}%
^2_1)$ we have that
\begin{align*}
\left( \gamma - \mathbb{D}_M \right)^{\beta }f(x) = & \frac{\beta}{%
\Gamma(1-\beta)} \int_0^\infty \frac{ds}{s^{\beta+1}} \left( 1- e^{-s\gamma
+ s \mathbb{D}_M}\right) f(x) \\
= & \frac{\beta}{\Gamma(1-\beta)} \int_0^\infty \frac{ds}{s^{\beta+1}}
\left( f(x) - e^{-s\gamma} \mathbb{P}_sf(x) \right)
\end{align*}
where, we recall that $\mathbb{P}_s f$ is the transition semigroup
associated with the operator $\mathbb{D}_M$ and $u(x, t) = \mathbb{P}%
_tf(x)$ solves the Cauchy problem $(\partial_t - \mathbb{D}_M)u(x,t)=0 $ with $u(x,0)=f(x)$. Therefore, if we assume that there
exists the following spectral representation for the solution $X$ as a random
function on $\mathbf{S}^2_1$,
\begin{equation}
X(x) = \sum_{lm} \hat{a}_{l,m} \mathcal{Y}_{l,m}(x)
\end{equation}
then we can immediately write
\begin{align*}
\left( \gamma - \mathbb{D}_M \right)^{\beta }X(x) = & \frac{\beta}{%
\Gamma(1-\beta)} \sum_{lm} \hat{a}_{l,m} \int_0^\infty \frac{ds}{s^{\beta+1}}
\left( \mathcal{Y}_{l,m}(x) - e^{-s \gamma} \mathbb{P}_s \mathcal{Y}_{l,m}(x)\right) \\
= & \frac{\beta}{\Gamma(1-\beta)} \sum_{lm} \hat{a}_{l,m} \int_0^\infty \frac{%
ds}{s^{\beta+1}} \left( 1- e^{-s \gamma} e^{-s \Psi(\mu_l)}\right) \mathcal{Y}_{l,m}(x)
\\
= & \sum_{lm} \hat{a}_{l,m} \left(\gamma + \Psi(\mu_l)\right)^{\beta
} \mathcal{Y}_{l,m}(x).
\end{align*}
The equation \eqref{gamma-D-X} turns out to be satisfied only if
\begin{equation*}
\hat{a}_{l,m} = a_{l,m} \, (\gamma + \Psi(\mu_l))^{-\beta}.
\end{equation*}
On the other hand, by repeating the arguments of the proof of Theorem \ref{thm-pro} we have that
\begin{align*}
X(x) = & \left( \gamma - \mathbb{D}_M \right)^{-\beta} T(x) \\
= & \int_0^\infty ds \frac{s^{\beta -1}}{\Gamma(\beta)} e^{-s\gamma + s
\mathbb{D}_M}T(x) \\
= & \int_0^\infty ds \frac{s^{\beta -1}}{\Gamma(\beta)} e^{-s\gamma} \mathbb{%
P}_s T(x) \\
= & \sum_{lm} a_{l,m} \int_0^\infty ds \frac{s^{\beta -1}}{\Gamma(\beta)}
e^{-s\gamma} e^{-s \Psi(\mu_l)} \mathcal{Y}_{l,m}(x) \\
= & \sum_{lm} a_{l,m} \left( \gamma + \Psi(\mu_l) \right)^{-\beta} \mathcal{Y}_{l,m}(x).
\end{align*}
This confirms result \eqref{confirm-sol}.
\end{proof}

We now study the covariance of the random fields introduced so far. Let us
consider the representation 
\begin{equation}
X_t(x) = \sum_{lm} a_{l,m} \mathcal{T}_l(t) \mathcal{Y}_{l,m}(x) = \sum_l 
\mathcal{T}_l(t) T_l(x)  \label{rep-X-gen}
\end{equation}
already introduced in Remark \ref{remark-form-convergence}. We also recall
that, for $x,y \in \mathbb{S}^2_1$, 
\begin{equation}
\mathbb{E}[X_0(x)\, X_0(y)] = \sum_l \frac{2l+1}{4\pi} \,C_l \, Q_l(\langle
x,y \rangle) = \mathbb{E}[T(x)\,T(y)].
\end{equation}
Furthermore, 
\begin{equation}
\mathbb{E}[T(x)\,T(y)] = \sum_l \mathbb{E}[T_l(x)\,T_l(y)].
\end{equation}
This is due to the fact that the coefficients $a_{l,m}$ are uncorrelated
over $l$.

\begin{remark}
We observe that
\begin{itemize}
\item for $X_{t}(x)$ as in Theorem \ref{thm:one}, 
\begin{equation}
\mathcal{T}_{l}(t)=E_{\beta }\left( -t^{\beta }(\gamma +\Psi (\mu
_{l}))\right) \geq \frac{1}{1+\Gamma (1-\beta )t^{\beta }(\gamma +\Psi (\mu
_{l}))},\quad t\geq 0,\;l\geq 0
\end{equation}%
For this inequality, consult \cite[Theorem 4]{simon14}.

\item for $X_t(x)$ as in Theorem \ref{thm:two}, 
\begin{equation}
\mathcal{T}_l(t) = e^{-t \mu_l} \left( \gamma + \varphi \mu_l + \Psi(\mu_l)
\right)^{-\beta} \leq e^{-t l^2} \left( \gamma + \varphi l^2 + \Psi(l^2)
\right)^{-\beta}, \quad t \geq 0,\; l \geq 0
\end{equation}

\item for $X(x)$ as in Corollary \ref{coro:one}, 
\begin{equation}
\mathcal{T}_l(0)= \left( \gamma + \Psi(\mu_l) \right)^{-\beta} \leq \left(
\gamma + \Psi(l^2) \right)^{-\beta} , \quad t \geq 0,\; l \geq 0
\end{equation}
\end{itemize}

\label{remarl-Trond}
\end{remark}

\begin{remark}
Let $B(\tau_t)$ be a time-changed Brownian motion on the unit sphere. We
refer to it as a coordinates change for the random field on the sphere $T$.
From the previous results, we observe that 
\begin{equation}
X_t(x) = \mathbb{E}[ T_0(x+ B(\tau_t)) | \mathfrak{F}_T] = \mathbb{E}%
[T_{\tau_t}(x)| \mathfrak{F}_T], \quad x \in \mathbb{S}^2_1, \; t>0
\end{equation}
where 
\begin{equation}
T_t(x) = \sum_l e^{-t \mu_l} T_l(x), \quad x \in \mathbb{S}^2_1, \; t>0.
\end{equation}
Moreover, 
\begin{equation}
\mathcal{T}_l(t) = \mathbb{E} e^{-\mu_l \tau_t}.  \label{lap-Trond-rep}
\end{equation}
\end{remark}

We can state the following result for which the spherical Brownian motions
underlying $X_t(x)$ and $X_t(y)$ are assumed independent.

\begin{thm}
For $x,y \in \mathbb{S}^2_1$, for all $g \in SO(3)$, we have that 
\begin{equation}
\mathbb{E}[X_{t}(gx)\, X_s(gy)] = \sum_l \frac{2l+1}{4\pi} \,C_l \, \mathcal{%
T}_l(t)\, \mathcal{T}_l(s)\, P_l(\langle x,y \rangle), \quad t,s\geq 0
\end{equation}
\end{thm}

\begin{proof}
First we observe that
\begin{equation}
\mathbb{E}[a_{l,m} a_{l^\prime, m^\prime}] = (-1)^m \delta_l^{l^\prime} \delta_{-m}^{m^\prime}C_l
\end{equation}
from the property $\mathcal{Y}_{l,m}(x) = (-1)^m \mathcal{Y}^*_{l-m}(x)$ of the spherical harmonics. From the representation \eqref{rep-X-gen} we can write
\begin{align*}
\mathbb{E}[X_t(x)\,X_s(y)] = & \sum_{lm} \sum_{l^\prime m^\prime} \mathbb{E}[a_{l,m} a_{l^\prime m^\prime}] \mathcal{T}_l(t) \mathcal{T}_{l^\prime}(s) \mathcal{Y}_{l,m}(x) \mathcal{Y}_{l^\prime m^\prime}(y) \\
= &  \sum_{lm} \, C_l\, \mathcal{T}_l(t) \mathcal{T}_{l}(s)\, \mathcal{Y}_{l,m}(x) \mathcal{Y}_{l,m}^*(y)\\
= &  \sum_{l} \frac{2l+1}{4\pi} C_l\, \mathcal{T}_l(t) \mathcal{T}_{l}(s)\, P_l(\langle x, y \rangle)
\end{align*}
where $\mathcal{T}_l(t)$ is given as in \eqref{lap-Trond-rep} and we used the addition formula in order to arrive at $P_l(\langle x, y \rangle)$. 
\end{proof}

\begin{remark}
We can immediately see that the variance becomes 
\begin{equation}
\mathbb{E}[X_{t}(gx)]^{2}=\sum_{l}\frac{2l+1}{4\pi }C_{l}\,|\mathcal{T}%
_{l}(t)|^{2},\quad \forall \,g\in SO(3).  \label{var-last}
\end{equation}%
We recall that $C_{l}$ is the angular power spectrum of $T$ and, it is
usually assumed to be $C_{l}\sim l^{-\gamma }$ with $\gamma \geq 2$ for
large $l$ to ensure summability (or $C_{l}\sim L(l)/t^{\theta }$ where $%
L(\cdot )$ is slowly varying function as $l\rightarrow \infty $). As Remark %
\ref{remarl-Trond} shows we have the high-frequency behaviour also for $%
\mathcal{T}_{l}(t)$ in both the variable $t>0$ and the frequency $l>0$. The
convergence of \eqref{var-last} therefore entails different correlation
structures for the solutions $X_{t}(x)$ of the equations investigated so far.
\end{remark}

We say that the zero mean process $X_t(x)$ exhibits a long range dependence
if 
\begin{equation}
\sum_{h=1}^\infty \mathbb{E}[X_{t+h}(x) X_t(y)] =\infty, \quad x,y, \in 
\mathbb{S}_1^2.  \label{long-condition}
\end{equation}
Conversely, we say that $X$ exhibits a short range dependence if the series %
\eqref{long-condition} converges.

\begin{remark}
We write 
\begin{equation*}
\mathcal{K}_t (x,y) = \sum_{h \geq 1} \mathbb{E}[X_{t+h}(x)\,X_t(y)], \quad
t\geq 0
\end{equation*}
for $x, y \in \mathbf{S}^2_1$. From the discussion above, we have that

\begin{itemize}
\item for $X_{t}(x)$ as in Theorem \ref{thm:one}, 
\begin{align}
\mathcal{K}_{t}(x,y)=& \sum_{h\geq 1}\sum_{l\geq 0}\frac{2l+1}{2\pi }%
C_{l}E_{\beta }(-t^{\beta }(\gamma +\Psi (\mu _{l})))E_{\beta
}(-(t+h)^{\beta }(\gamma +\Psi (\mu _{l}))) \\
\geq & \sum_{h\geq 1}\sum_{l\geq 0}\frac{2l+1}{2\pi }C_{l}\frac{1}{1+\Gamma
(1-\beta )t^{\beta }(\gamma +\Psi (\mu _{l}))}\frac{1}{1+\Gamma (1-\beta
)(t+h)^{\beta }(\gamma +\Psi (\mu _{l}))} \notag \\
\geq & \sum_{l\geq 0}\frac{2l+1}{2\pi }C_{l}\frac{1}{1+\Gamma (1-\beta
)t^{\beta }(\gamma +\Psi (\mu _{l}))}\sum_{h\geq 1}\frac{1}{1+\Gamma
(1-\beta )(t+h)^{\beta }(\gamma +\Psi (\mu _{l}))} \notag \\
=& \infty, \notag
\end{align}%
that is the random field exhibits a long-range dependence;
\item for $X_{t}(x)$ as in Theorem \ref{thm:two}, 
\begin{align}
\mathcal{K}_{t}(x,y)\leq & \sum_{h\geq 1}\sum_{l\geq 0}\frac{2l+1}{4\pi }%
C_{l}\,e^{-2tl^{2}-hl^{2}}\left( \gamma +\varphi l^{2}+\Psi (l^{2})\right)
^{-2\beta }\,P_{l}(\langle x,y\rangle )  \\
\leq & \sum_{l\geq 0}\frac{2l+1}{4\pi }\frac{C_{l}}{e^{l^{2}}-1}%
e^{-2tl^{2}}\left( \gamma +\varphi l^{2}+\Psi (l^{2})\right) ^{-2\beta
}\,P_{l}(\langle x,y\rangle )  \notag \\
<& \infty, \notag
\end{align}%
that is the random field has a short range dependence.
\end{itemize}

\label{remarl-Trond-memory-t}
\end{remark}

\end{document}